\documentclass{amsart}

\usepackage{amsmath,amsthm,amsfonts,amssymb}
\usepackage{mathtools}
\usepackage{rotating}
\usepackage{mathdots}
\usepackage[usenames,dvipsnames]{xcolor}

\usepackage{setspace}
\usepackage{epigraph}
\newtheorem*{theorem*}{Theorem}
\newtheorem{lemma}{Lemma}

\newtheorem*{corollary*}{Corollary}

\newtheorem*{remark*}{Remark}

\newtheorem*{definition*}{Definition}

\title{Products of Sylow subgroups\\in Suzuki and Ree groups}
\author{Andrei Smolensky}
\address{Department of Mathematics and Mechanics, Saint Petersburg State University,\\
Universitetsky prospekt, 28, 198504, Peterhof, Saint Petersburg, Russia}
\email{andrei.smolensky@gmail.com}
\date{\today}

\begin{document}
\begin{abstract}
An explicit and elementary proof is given to the fact that Suzuki and Ree groups can be decomposed into the product of $4$ of their Sylow $p$-subgroups, where $p$ is the defining characterictic.
\end{abstract}
\maketitle{}

The problem of factorizing a finite simple group of Lie type into the product of its Sylow $p$-subgroups, where $p$ is the defining characteristic, has recently got some attention. For example, in \cite{BabNikPybProdGr} it is claimed that each of them is a product of $5$ Sylow subgroups, and in \cite{LiebeckPyberFiniteLinear} an estimate of $13$ factors is given for groups other than ${}^2\mathsf{F}_4$, and of $25$ for the latter. It has been used, for example, in \cite{KasLubNikFinisigExpanders} and \cite{GarLevConjugFactor}.

In \cite{VavSmSuUnitrEng} a simple proof is given to the fact that Chevalley groups of normal types admit the factorization of length $4$, and so do twisted groups of types ${}^2\mathsf{A}_{2n+1}$, ${}^2\mathsf{D}_n$, ${}^3\mathsf{D}_4$ and ${}^2\mathsf{E}_6$. The proof is based on Tavgen rank reduction theorem \cite{TavgenBoundGen,TavgenBoundGenProc}, which doesn't care about the base ring at all, and the decomposition of $SL_2$, which can be carried over any ring of stable rank $1$. A slightly more technical case of ${}^2\mathsf{A}_{2n}$ is elaborated in \cite{SmoUnitrTwisted} by explicitly factorizing the group of type ${}^2\mathsf{A}_2$.

In the present note we analyze the remaining cases of so-called very twisted groups, namely Suzuki group ${}^2\mathsf{C}_2$, small Ree group ${}^2\mathsf{G}_2$ and large Ree group ${}^2\mathsf{F}_4$.

Together with the previously known results they give the following
\begin{theorem*}
Every finite simple group of Lie type in characteristic $p$ is a product of $4$ of its Sylow $p$-subgroups.
\end{theorem*}

\section{Suzuki group and large Ree group}
In what follows Suzuki group is considered as a subgroup of symplectic group over a field $F=\mathbb{F}_q$, $q=2^{2m+1}$, pointwise invariant under the action of the exceptional automorphism. The symplectic group is defined as a subgroup of $GL(4,F)$, preservinig a symplectic form with the matrix $(\delta_{i,-j})_{i,j=1\ldots-1}$. Here, as usual, we number the set of indices as $1,2,-2,-1$.

We also define elementary symplectic transvections $T_{ij}(\xi)$, $i\neq j$, $\xi\in F$ of two types as follows:
\begin{align*}
&T_{ij}(\xi) = T_{-j,-i}(\xi) = e+\xi e_{ij} + \xi e_{-j,-i},\ i\neq\pm j\\
&T_{i,-i}(\xi) = e+\xi e_{i,-i},
\end{align*}
They are called respectively short and long elementary root unipotents, according to their identification with the elementary root unipotents in the sense of Chevalley groups, namely
\begin{align*}
&T_{1,2}(\xi) = x_{\alpha}(\xi), & &T_{2,-2}(\xi) = x_{\beta}(\xi), \\ &T_{1,-2}(\xi) = x_{\alpha+\beta}(\xi), & &T_{1,-1}(\xi) =  x_{2\alpha+\beta}(\xi),
\end{align*}
where $\{\alpha,\beta,\alpha+\beta,2\alpha+\beta\}$ are positive roots of $\mathsf{C}_2$.

Denote $\theta=2^m$, so that $q=2\theta^2$, and consider the automorphism $t\mapsto t^\theta$ of $F$. It allows to define the exceptional automorphism of $Sp(4,F)$, corresponding to the symmetry $\alpha\mapsto\overline{\alpha}$ of its Dynkin diagram, as follows
\[ \sigma\colon x_\alpha\left(\xi\right)\longmapsto x_{\overline{\alpha}}\left(\xi^{\lambda(\overline{\alpha})\theta}\right),\text{ where } \lambda(\beta)=\begin{cases}
1 & \text{if $\beta$ is short,} \\ 2 & \text{if $\beta$ is long.}
\end{cases} \]
We define Suzuki group as $Sz(q)={}^2\mathsf{C}_2\left(2^{2m+1}\right)=\left\{g\in Sp(4,F)\mid \sigma(g)=g\right\}$.

This definition is hardly suitable for computations, as it doesn't exhibit any equations on the matrix entries, except for those coming from the symplectic group. Fortunately, Suzuki group inherits Bruhat decomposition $Sz(q)=\bigsqcup BwU$, where $w$ runs over its Weyl group. It can be taken as the definition.

Define $x_+(t,u)=x_\alpha(t^\theta)\,x_\beta(t)\,x_{\alpha+\beta}(t^{\theta+1}+u)\,x_{2\alpha+\beta}(u^{2\theta})$ for any $t,u\in F$. They are exactly the elements of $U^+(\mathsf{C}_2,F)$, stable under the action of $\sigma$. The subgroup $U=\{x_+(t,u)\mid t,u\in F\}$ is a Sylow $2$-subgroup of $Sz(q)$. As matrices
\[ x_+(t,u) =
\begin{pmatrix}
1 & t^\theta & u & t^{2\theta+1}+t^\theta u+u^{2\theta} \\
& 1 & t & t^{\theta+1}+u \\
& & 1 & t^\theta \\
& & & 1
\end{pmatrix}. \]
The elements of the torus, stable under $\sigma$, are of the form \[ h(\varepsilon)=\operatorname{diag}\left(\varepsilon,\varepsilon^{2\theta-1},\varepsilon^{1-2\theta},\varepsilon^{-1}\right). \]
Writing $T=\{ h(\varepsilon)\mid\varepsilon\in F^* \}$ for the torus, we denote by $B=UT$ a Borel subgroup. The elements of the torus normalizer $N$ take one of the following two forms:
\begin{align*}
h(\varepsilon)& =
\begin{pmatrix}
\varepsilon &&& \\ & \varepsilon^{2\theta-1} && \\ && \varepsilon^{1-2\theta} & \\ &&& \varepsilon^{-1}
\end{pmatrix},& \\
h(\varepsilon) w& =
\begin{pmatrix}
&&& \varepsilon \\ && \varepsilon^{2\theta-1} & \\ & \varepsilon^{1-2\theta} && \\ \varepsilon^{-1} &&&
\end{pmatrix},&
\text{ where\quad} w=
\begin{pmatrix}
&&& 1 \\ && 1 & \\ & 1 && \\ 1 &&&
\end{pmatrix}.
\end{align*}
Bruhat decomposition states that $Sz(q)=BNU$, or that every element $g\in Sz(q)$ can be uniquely decomposed into one of the following two products:
\[ g=x_+(t_1,u_1)\,h(\varepsilon)\,w\,x_+(t_2,u_2),\quad
g=x_+(t,u)\,h(\varepsilon). \]
Now we define $x_-(t,u)=w\,x_+(t,u)\,w$ and $U^-=\{x_-(t,u)\mid t,u\in F\}$. The subgroup $U^-={}^wU$ is again $2$-Sylow.
\begin{lemma}\label{lemma:suz}
$Sz(q)=UU^-UU^-$.
\end{lemma}
\begin{proof}
In view of Bruhat decomposition it is sufficient to show that $Tw\subset UU^-U$ and $T\subset UU^-UU^-$.

For the first one fix an element $\varepsilon\in F^*$ and consider $h(\varepsilon)w$. Write
\[ x_-\left(\varepsilon^{1-2\theta},0\right)\,x_+\left(0,\varepsilon^\theta\right)=
\begin{pmatrix}
1 & 0 & \varepsilon^\theta & \varepsilon \\
\varepsilon^{\theta-1} & 1 & \varepsilon^{2\theta-1} \\
\varepsilon^{-\theta} & \varepsilon^{1-2\theta} & & \\
\varepsilon^{-1} & & &
\end{pmatrix}=
x_+\left(\varepsilon^{2\theta-1},0\right)\,h(\varepsilon)w. \]
The decomposition for the torus is suggested by the decomposition of the torus in $SL(2,R)$:
\[ \begin{pmatrix}
\varepsilon & 0 \\ 0 & \varepsilon^{-1}
\end{pmatrix} =
\begin{pmatrix}
1 & -1 \\ 0 & 1
\end{pmatrix}
\begin{pmatrix}
1 & 0 \\ 1-\varepsilon & 1
\end{pmatrix}
\begin{pmatrix}
1 & \varepsilon^{-1} \\ 0 & 1
\end{pmatrix}
\begin{pmatrix}
1 & 0 \\ \varepsilon(\varepsilon-1) & 1
\end{pmatrix}. \]
Consider the following product (note the entry in the upper left corner):
\begin{align*}
g_1 & =  x_+\left(0,1\right)\cdot x_-\left(1,1+\varepsilon^\theta\right)=
\begin{pmatrix}
\colorbox{Gray!20}{$\varepsilon$} & \varepsilon^\theta & 0 & 1 \\
\varepsilon+\varepsilon^\theta & \varepsilon^\theta & 1 & 1 \\
\varepsilon^\theta & 1 & 1 & 0 \\
1+\varepsilon+\varepsilon^\theta & 1+\varepsilon^\theta & 1 & 1
\end{pmatrix},\\
g_2 & = g_1\cdot x_+\left(\varepsilon^{1-2\theta},\varepsilon^{-\theta}\right) =\\
& = \begin{pmatrix}
\varepsilon & 0 & 0 & 0 \\
\varepsilon+\varepsilon^\theta & \varepsilon^{2\theta-1} & 0 & 0 \\
\varepsilon^\theta & 1+\varepsilon^{2\theta-1} & \varepsilon^{1-2\theta} & 0 \\
1+\varepsilon+\varepsilon^\theta & 1+\varepsilon^{\theta-1}+\varepsilon^{2\theta-1} & \varepsilon^{-\theta}+\varepsilon^{1-2\theta} & \varepsilon^{-1}
\end{pmatrix}.
\end{align*}
One may even not bother himself with computing the entries of $g_2$ below the diagonal and just note that $g_2$ is of the form $h(\varepsilon)\,x_-(t,u)$ for some $t,u\in F$, as required. In fact $t=\varepsilon^{2\theta-1}(1+\varepsilon^{2\theta-1})$ and $u=\varepsilon+\varepsilon^\theta+\varepsilon^{2\theta}$.
\end{proof}
The large Ree group is considered as a subgroup of the Chevalley group $G(\mathsf{F}_4,\mathbb{F}_q)$, $q=2^{2m+1}$, defined by the exceptional symmetry of its Dynkin diagram in the same way as Suzuki group. It is not a rank $1$ group, so we simply refer to the Tavgen theorem (which works also for very twisted groups), reducing the question to the already known cases of groups of types $\mathsf{A}_1$ and ${}^2\mathsf{C}_2$, thus giving
\begin{corollary*}
${}^2\mathsf{F}_4\left(2^{2m+1}\right)=UU^-UU^-$.
\end{corollary*}
\section{Small Ree group}
\epigraph{I'm now trying to convince my students that they shouldn't multiply two 3-by-3 matrices by hand: just like multiplying two multidigit numbers, it doesn't make any sense.}{N.~Vavilov at PhML-2014}
We work with small Ree group as a subgroup of $G(\mathsf{G}_2,F)$ for a field $F=\mathbb{F}_q$, $q=3^{2m+1}=3\theta^2$. $G(\mathsf{G}_2,F)$ is considered in its 7-dimensional representation, and the elementary generators are chosen as follows (the basis is numbered as $1,2,3,0,-3,-2,-1$):
\begin{align*}
\begin{sideways}\text{short}\end{sideways} & \left[
\begin{aligned}
&x_{\alpha}(\xi)=e+\xi(e_{12}-e_{30}+2e_{0,-3}-e_{-2,-1})-\xi^2e_{3,-3},\\
&x_{\alpha+\beta}(\xi)=e+\xi(e_{13}+e_{20}-2e_{0,-2}-e_{-3,-1})-\xi^2e_{2,-2},\\
&x_{2\alpha+\beta}(\xi)=e+\xi(-e_{10}+e_{2,-3}-e_{3,-2}+2e_{0,-1})-\xi^2e_{1,-1},\\
&x_{-\alpha}(\xi)=e+\xi(e_{21}-2e_{03}+e_{-3,0}-e_{-1,-2})-\xi^2e_{-3,3},\\
&x_{-\alpha-\beta}(\xi)=e+\xi(e_{31}+2e_{02}-e_{-2,0}-e_{-1,-3})-\xi^2e_{-2,2},\\
&x_{-2\alpha-\beta}(\xi)=e+\xi(-2e_{01}+e_{-3,2}-e_{-2,3}+e_{-1,0})-\xi^2e_{-1,1},\\
\end{aligned}\right.\\
\begin{sideways}\text{long}\end{sideways} & \left[
\begin{aligned}
&x_{\beta}(\xi)=e+\xi(-e_{23}+e_{-3,-2})=x_{-\beta}(\xi)^t,\\
&x_{3\alpha+\beta}(\xi)=e+\xi(e_{1,-3}-e_{3,-1})=x_{-3\alpha+\beta}(\xi)^t,\\
&x_{3\alpha+2\beta}(\xi)=e+\xi(-e_{1,-2}+e_{2,-1})=x_{-3\alpha-2\beta}(\xi)^t.
\end{aligned}\right.
\end{align*}
This particular choice of the signs of the structure constants coincides with one in \cite[Section~12.4]{CarterLie}, thus allowing to define an automorphism of $G(\mathsf{G}_2,F)$ by
\[ \sigma\colon x_\alpha\left(\xi\right)\longmapsto x_{\overline{\alpha}}\left(\xi^{\lambda(\overline{\alpha})\theta}\right),\text{ where } \lambda(\beta)=\begin{cases}
1, & \text{if $\beta$ is short,} \\ 3, & \text{if $\beta$ is long.}
\end{cases} \]
Small Ree group ${}^2\mathsf{G}_2\left(3^{2m+1}\right)$ is then defined as the subgroup of points fixed by $\sigma$.

Elements of $U^+(\mathsf{G}_2,F)$ fixed by $\sigma$ are of the form
\begin{align*}
& x_+(t,u,v)=x_1(t)x_2(u)x_3(v),\text{ where }t,u,v\in F\text{ and}\\
&\quad x_1(t)= x_\alpha\left(t^\theta\right)\,x_\beta\left(t\right)\,x_{\alpha+\beta}\left(t^{\theta+1}\right)\,x_{2\alpha+\beta}\left(t^{2\theta+1}\right),\\
&\quad x_2(u)=x_{\alpha+\beta}\left(u^\theta\right)\,x_{3\alpha+\beta}\left(u\right),\\
&\quad x_3(v)=x_{2\alpha+\beta}\left(v^\theta\right)\,x_{3\alpha+2\beta}\left(v\right).
\end{align*}
$U^+={}^\sigma U^+(\mathsf{G}_2,F)$ is a Sylow $3$-subgroup ${}^2\mathsf{G}_2\left(3^{2m+1}\right)$.

The invariant elements of the torus are of the form $h(\varepsilon)=h_\alpha\left(\varepsilon\right)h_\beta\left(\varepsilon^{3\theta}\right)$. Denote by $w=\left(w_\alpha(1)w_\beta(1)\right)^3$ a lift of the longest element of $W(\mathsf{G}_2)$. It lies in ${}^2\mathsf{G}_2$ and represents the nontrivial element of its Weyl group. As matrices
\[ h\left(\varepsilon\right)= \operatorname{diag}\left(\varepsilon,\varepsilon^{3\theta-1},\varepsilon^{2-3\theta},1,\varepsilon^{3\theta-2},\varepsilon^{1-3\theta},\varepsilon^{-1}\right),\quad w=\begin{pmatrix}
&&-1 \\ &\iddots& \\ -1&&
\end{pmatrix}. \]
Bruhat decomposition states that every element $g$ of small Ree group can be written uniquely in one of the two forms
\[ g=x_+(t_1,u_1,v_1)\,h(\varepsilon)\,w\,x_+(t_2,u_2,v_2),\quad g=x_+(t,u,v)\,h(\varepsilon). \]
\begin{lemma}\label{lemma:ree}
${}^2\mathsf{G}_2\left(3^{2m+1}\right)=UU^-UU^-$.
\end{lemma}
\begin{proof}
First consider $h(\varepsilon)w$ for some $\varepsilon\in F^*$. Any element of $F=\mathbb{F}_{3\theta^2}$ is either a square or a minus square. If $\varepsilon=\lambda^2$, write
\begin{multline*}
x_-\left(-\lambda^{3-6\theta},0,\lambda^{-3\theta}\right)\,x_+\left(0,0,\lambda^{3\theta}\right)=\\
=\begin{pmatrix}
1 & 0 & 0 & -\lambda & 0 & -\lambda^{3\theta} & -\lambda^2 \\
\lambda^{3\theta-2} & 1 & 0 & -\lambda^{3\theta-1} & \lambda & -\lambda^{6\theta-2} & 0 \\
-\lambda^{1-3\theta} & -\lambda^{3-6\theta} & 1 & \lambda^{2-3\theta} & -\lambda^{4-6\theta} & 0 & 0 \\
-\lambda^{-1} & 0 & \lambda^{3\theta-2} & -1 & 0 & 0 & 0 \\
-\lambda^{3\theta-3} & -\lambda^{-1} & -\lambda^{6\theta-4} & 0 & 0 & 0 & 0 \\
\lambda^{-3\theta} & -\lambda^{2-6\theta} & 0 & 0 & 0 & 0 & 0 \\
-\lambda^{-2} & 0 & 0 & 0 & 0 & 0 & 0
\end{pmatrix}=\\
= x_+\left(\lambda^{6\theta-3},0,-\lambda^{3\theta}\right)\, h\left(\lambda^2\right)w.
\end{multline*}
If $\varepsilon=-\lambda^2$, write
\begin{multline*}
x_-\left(-\lambda^{3-6\theta},-\lambda^{3\theta-3},\lambda^{-3\theta}\right)\,x_+\left(\lambda^{6\theta-3},0,0\right)=\\
=\begin{pmatrix}
1 & \lambda^{2-3\theta} & 0 & 0 & \lambda^{3-3\theta} & -\lambda^{3\theta} & \lambda^2 \\
\lambda^{3\theta-2} & -1 & -\lambda^{6\theta-3} & \lambda^{3\theta-1} & -\lambda & -\lambda^{6\theta-2} & 0 \\
0 & -\lambda^{3-6\theta} & -1 & \lambda^{2-3\theta} & \lambda^{4-6\theta} & 0 & 0 \\
0 & \lambda^{1-3\theta} & -\lambda^{3\theta-2} & -1 & 0 & 0 & 0 \\
\lambda^{3\theta-3} & -\lambda^{-1} & \lambda^{6\theta-4} & 0 & 0 & 0 & 0 \\
-\lambda^{-3\theta} & -\lambda^{2-6\theta} & 0 & 0 & 0 & 0 & 0 \\
\lambda^{-2} & 0 & 0 & 0 & 0 & 0 & 0
\end{pmatrix}=\\
=x_+\left(\lambda^{6\theta-3},\lambda^{3-3\theta},\lambda^{3\theta}\right)\,h\left(-\lambda^2\right)w.
\end{multline*}
Now consider an element $h(\varepsilon)$, $\varepsilon\in F^*$. Again, if $\varepsilon=\lambda^2$, let
\begin{multline*}
g_1=x_+\left(0,1,0\right)\,x_-\left(0,1,\lambda^{3\theta}\right)=\\
=\begin{pmatrix}
\colorbox{Gray!20}{$\lambda^2$} & \lambda^{3\theta} & 0 & \lambda & 0 & 0 & -1 \\
-\lambda^{3\theta} & 0 & -\lambda & 0 & 0 & -1 & 0 \\
\lambda^2 & \lambda^{3\theta}+\lambda & 0 & \lambda & -1 & 0 & -1 \\
\lambda^{3\theta}+\lambda & 0 & \lambda & -1 & 0 & 1 & 0 \\
\lambda^2 & \lambda^{3\theta} & -1 & \lambda & 0 & 0 & -1 \\
\lambda^{3\theta}-\lambda & -1 & \lambda & 1 & 0 & 1 & 0 \\
-1-\lambda^2 & -\lambda^{3\theta}-\lambda & 1 & -\lambda & 1 & 0 & 1
\end{pmatrix}.
\end{multline*}
Multiplication by $x_+\left(-\lambda^{3-6\theta},-\lambda^{3\theta-3},-\lambda^{-3\theta}\right)$ on the right eliminates all the entries above the diagonal, which gives
\begin{multline*}
g_1\cdot x_+\left(-\lambda^{3-6\theta},-\lambda^{3\theta-3},-\lambda^{-3\theta}\right)=\\
=h\left(\lambda^2\right)\,x_-\left(\lambda^{6\theta-3},-\lambda^{3-3\theta}(1+\lambda^{3-3\theta}),\lambda^{3\theta}+\lambda\right).
\end{multline*}
If $\varepsilon=-\lambda^2$, put
\begin{multline*}
g_1=x_+\left(-1,-1,1\right)\,x_-\left(1,0,\lambda^{3\theta}\right)=\\
=\begin{pmatrix}
\colorbox{Gray!20}{$-\lambda^2$} & -\lambda^{3\theta} & 0 & -\lambda & 0 & 0 & 1 \\
\lambda^2-\lambda^{3\theta} & \lambda^{3\theta} & -\lambda & \lambda & 0 & -1 & -1 \\
-\lambda^2-\lambda^{3\theta} & -\lambda-\lambda^{3\theta} & -\lambda & -\lambda & 1 & -1 & 1 \\
\lambda-\lambda^{3\theta} & \lambda & -\lambda & -1 & -1 & -1 & 0 \\
-\lambda-\lambda^{3\theta} & \lambda & 1-\lambda & 1 & -1 & -1 & 0 \\
-\lambda^2-\lambda-\lambda^{3\theta} & \lambda-\lambda^{3\theta}-1 & -1-\lambda & 1-\lambda & -1 & -1 & 1 \\
1+\lambda^{3\theta}-\lambda^2 & -\lambda^{3\theta}-1 & 1+\lambda & -\lambda & 0 & 1 & 1
\end{pmatrix}.
\end{multline*}
A straightforward calculation shows that
\[
g_1\cdot x_+\left(-\lambda^{3-6\theta},-\lambda^{3\theta-3},-\lambda^{-3\theta}\right)=h\left(-\lambda^2\right)\,x_-\left(u,v,uv\right),\]
where $u=\lambda^{6\theta-3}\left(\lambda^{6\theta-3}-1\right)$ and $v=-u^{3\theta+1}$.
\end{proof}
\section*{Addendum}
Immediately after the publication of the present paper another paper \cite{GarLevMarSimSylowWidth} appeared, containing a proof of the same result as out main Theorem. The proof goes along the same line, with the following difference:
\begin{itemize}
\item The rank reduction is very similar to that of Tavgen~\cite{TavgenBoundGen,TavgenBoundGenProc}, but reduced directly to the groups of rank $1$ rather than the subgroups of submaximal rank.
\item The analysis of rank $1$ groups consists of two steps. The first is Lemma~5, which shows that the factorisation of length $4$ exists when every lift of the nontrivial element of the Weyl group to the torus normalizer can be decomposed into the product of three unitriangular factors.
\item The second is deducing the existence of that decomposition from the result of Proposition~4.1 of \cite{CheEllGorGaussDecomp}. The latter is proved case-by-case, and the computation is essentially the same as in our Lemmas~\ref{lemma:suz} and \ref{lemma:ree}.
\end{itemize}
It is now clear that the shortest and simplest proof would be the following: 1)~Tavgen rank reduction theorem; 2)~Lemma~5 of \cite{GarLevMarSimSylowWidth}; 3)~the part of Lemmas~\ref{lemma:suz} and \ref{lemma:ree} dealing with the lift of the nontrivial element of the Weyl group, and the same proof for $\mathsf{A}_1$ and ${}^2\mathsf{A}_2$.
\section*{Acknowledgments}
Research was supported by RFFI (grants 12-01-00947-a and 14-01-00820) and by State Financed task project 6.38.191.2014 at Saint Petersburg State University.
\bibliographystyle{amsalpha}
\bibliography{suzuki-ree-sylow}
\end{document}